\documentclass[12pt]{amsart}
\usepackage{amsmath,amsthm,latexsym,amscd,amsbsy,amssymb,amsfonts,fleqn,leqno,
euscript, pb-diagram}

\numberwithin{equation}{section}

\newcommand{\I}{\mathbb I}

\newcommand{\e}{\varepsilon}

\newtheorem{thm}{Theorem}[section]
\newtheorem{pro}[thm]{Proposition}
\newtheorem{lem}[thm]{Lemma}
\newtheorem{cor}[thm]{Corollary}

\begin{document}

\title[Krasinkiewicz spaces and parametric Krasinkiewicz maps]
{Krasinkiewicz spaces and parametric Krasinkiewicz maps}
\author{Eiichi Matsuhashi}
\address{Department of Mathematics, Faculty of Engineering,
Yokohama National University Yokohama, 240-8501, Japan}
\email{mateii@ynu.ac.jp}
\author{Vesko Valov}
\address{Department of Computer Science and Mathematics,
Nipissing University, 100 College Drive, P.O. Box 5002, North Bay,
ON, P1B 8L7, Canada}
\email{veskov@nipissingu.ca}
\thanks{The second author was partially supported by NSERC
Grant 261914-03.} 

\begin{abstract}
We say that a metrizable space $M$ is a Krasinkiewicz space if any
map from a metrizable compactum $X$ into $M$ can be approximated by
Krasinkiewicz maps (a map $g\colon X\to M$ is Krasinkiewicz provided
every continuum in $X$ is either contained in a fiber of $g$ or
contains a component of a fiber of $g$). In this paper we establish
the following property of Krasinkiewicz spaces: Let $f\colon X\to Y$
be a perfect map between metrizable spaces and $M$ a Krasinkiewicz
complete $ANR$-space. If $Y$ is a countable union of closed
finite-dimensional subsets, then the function space $C(X,M)$ with
the source limitation topology contains a dense $G_{\delta}$-subset
of maps $g$ such that all restrictions $g|f^{-1}(y)$, $y\in Y$, are
Krasinkiewicz maps. The same conclusion remains true if $M$ is
homeomorphic to a closed convex subset of a Banach space and $X$ is
a $C$-space.
\end{abstract}
\maketitle \markboth{}{Krasinkiewicz spaces and maps}


\section{Introduction}
All spaces in the paper are assumed to be metrizable and all maps
continuous. Unless stated otherwise, any function space $C(X,M)$ is
endowed with the {\em source limitation topology}. This topology,
known also as the {\em fine topology}, was introduced by Whitney
\cite{w} and has a base at a given $f\in C(X,M)$ consisting of the
sets
$$B_\varrho(f,\e)=\{g\in C(X,M):\varrho(g,f)<\e\},$$
where $\varrho$ is a fixed compatible metric on $M$ and
$\e:X\to(0,1]$ runs over continuous functions into $(0,1]$. The
symbol $\varrho(f,g)<\e$ means that
$\varrho\big(f(x),g(x)\big)<\e(x)$ for all $x\in X$. The source
limitation topology doesn't depend on the metric $\varrho$ \cite{nk}
and has the Baire property provided $M$ is completely metrizable
\cite{munkers}. Obviously, this topology coincides with the uniform
convergence topology when $X$ is compact.

We say that a space $M$ is a {\em Krasinkiewicz space} if for any
compactum $X$ the function space $C(X,M)$ contains a dense subset of
Krasinkiewicz maps. Recall that a map $g\colon X\to M$, where $X$ is
compact, is said to be Krasinkiewicz \cite{ll} if every continuum in
$X$ is either contained in a fiber of $g$ or contains a component of
a fiber of $g$. Krasinkiewicz \cite{kr1} proved that every
$1$-manifold is a Krasinkiewicz space (for the interval $\I$ this
was established by Levin-Lewis \cite{ll}). The first author,
generalizing the  Krasinkiewicz result, proved in \cite{m} that all
compact polyhedra, as well as all $1$-dimensional Peano continua and
manifolds modeled on a Menger cube are Krasinkiewicz spaces.

The main results in this paper is the following theorem:

\begin{thm} Let $M$ be a Krasinkiewicz complete $ANR$-space and $f\colon X\to
Y$ a perfect map with $Y$ being a strongly countable-dimensional
space. Then the function space $C(X,M)$ contains a dense
$G_\delta$-set of maps $g$ such that all restrictions $g|f^{-1}(y)$,
$y\in Y$, are Krasinkiewicz maps. Moreover, if in addition $M$ is a
closed convex subset of a Banach space, then the same conclusion
remains true provided $Y$ is a $C$-space.
\end{thm}

 Recall that $X$ is a $C$-space if for any sequence
$\{\nu_n\}_{n=1}^{\infty}$ of open covers of $X$ there exists a
sequence $\{\gamma_n\}_{n=1}^{\infty}$ of disjoint open families in
$X$ such that each $\gamma_n$ refines $\nu_n$ and
$\cup_{n=1}^{\infty}\gamma_n$ is a cover of $X$. Every strongly
countable-dimensional space (i.e. a space which is a union of
countably many closed finite-dimensional subsets), as well as every
countable-dimensional space (a countable union of 0-dimensional
subsets) is a $C$-space \cite{re:95} and there exists a compact
$C$-space which is not countable-dimensional.

Everywhere below by a {\em polyhedron} we mean the underlying space
of a simplicial complex equipped with the metric topology. A
compactum is called a {\em Bing space} if each of its subcontinua is
hereditarily indecomposable. According to Corollary 3.2, each
polyhedron is a Krasinkiewicz space. Moreover, it follows from
\cite{st} that for any polyhedron $P$ without isolated points and a
compactum $X$ the space $C(X,P)$ contains a dense set of {\em Bing
maps} (maps $g$ such that all fibers $g^{-1}(y)$, $y\in P$, are Bing
spaces). Therefore, Theorem 1.1 and \cite[Theorem 1.1]{v} imply the
following corollary:

\begin{cor}
Let $P$ be a complete polyhedron without isolated points and
$f\colon X\to Y$ a perfect map. Then the function space $C(X,P)$
contains a dense $G_\delta$-set of maps $g$ such that all
restrictions $g|f^{-1}(y)$, $y\in Y$, are both Bing and
Krasinkiewicz maps in each of the following cases: $(i)$ $Y$ is
strongly countable-dimensional; $(ii)$ $Y$ is a $C$-space and $P$ is
a closed convex subset of a Banach space.
\end{cor}

Most part of the paper is devoted to the proof of Theorem 1.1, given
in Section 2. In Section 3 we provide some properties of
Krasinkiewicz spaces. For example, we show that a complete $ANR$ is
a Krasinkiewicz space if and only if it has an open cover of
Krasinkiewicz subspaces. In particular, all $n$-manifolds, $n\geq
1$, are Krasinkiewicz spaces.

\section{Proof of Theorem 1.1}
We fixed a metric $d$ on $X$ and for every $A\subset X$ and
$\delta>0$ let $B(A,\delta)=\{x\in X: d(x,A)<\delta\}$. If $y\in Y$
and $m,n\geq 1$, then $\mathcal{K}(m,n,y)$ denotes the set of all
maps $g\in C(X,M)$ satisfying the following condition:

\begin{itemize}
\item For each subcontinuum $L\subset f^{-1}(y)$ with diam$g(L)\ge 1/n$
there exists $x\in L$ such that $C(x,g|f^{-1}(y))\subset B(L,1/m)$.
Here, $g|f^{-1}(y)$ is the restriction of $g$ over $f^{-1}(y)$ and
$C(x,g|f^{-1}(y))$ denotes the component of the fiber
$g^{-1}(g(x))\cap f^{-1}(y)$ of $g|f^{-1}(y)$ containing $x$.
\end{itemize}

For $H\subset Y$ let $\mathcal{K}(m,n,H)$ be the intersection of all
$\mathcal{K}(m,n,y)$, $y\in H$.  We also denote by $\mathcal{K}(H)$
the set of all maps $g\in C(X,M)$ such that $g|f^{-1}(y)\colon
f^{-1}(y)\to M$ is a Krasinkiewicz map for each $y\in H$.

\begin{pro}
$\mathcal{K}(H)=\bigcap_{m,n \in \mathbb{N}} K(m,n,H)$.
\end{pro}

\begin{proof} Obviously $\mathcal{K}(H) \subset \bigcap_{m,n \in \mathbb{N}}
K(m,n,H)$. So, we need to prove the inclusion
 $\bigcap_{m,n \in \mathbb{N}} K(m,n,H)  \subset  \mathcal{K}(H)$.
Let $g \in \bigcap_{m,n \in \mathbb{N}} K(m,n,H)$, $y \in H$ and $L
\subset f^{-1}(y)$ be a subcontinuum such that diam$g(L) >0$. We are
going to prove that there exists a subcontinuum $L_2 \subset L_1=L$
such that diam$g(L_2)>0$ and $C(x,g|f^{-1}(y)) \subset B(L_1,1/2)$
for each $x \in L_2$. Since diam$g(L_1)>0$, there exists $n_1 \in
\mathbb{N}$ such that diam$g(L_1) \ge 1/n_1$. Since $g \in
K(2,n_1,y)$, there exists a point $x \in L_1$ such that
$C(x,g|f^{-1}(y)) \subset B(L_1, 1/2)$. Let $E \subset L_1$ be the
set of all such points. It is easy to see that:

\begin{itemize}
\item[($\sharp$)] every  $x \in E$ has a neighborhood $U_x$
in $L_1$ with $C(z,g|f^{-1}(y)) \subset B(L_1,1/2)$ for all $z \in
U_x$.
\end{itemize}

Let $x_0 \in E$ and $\mathcal{D}$ be the family of all subcontinua
$D$ of $L_1$ such that  $x_0 \in D$ and $C(d,g|f^{-1}(y)) \subset
B(L_1,1/2)$ for each $d \in D$. Since $\{x_0\} \in \mathcal{D}$,
$\mathcal{D} \neq \emptyset$.

\medskip
\textit{Claim. There exists $D^* \in \mathcal{D}$ such that diam
$g(D^*)>0$}.

\vspace{2mm}

Assume that $g(D)$ is a singleton for each $D \in \mathcal{D}$. Then
cl$(\bigcup \mathcal{D}) \in \mathcal{D}$. In fact, if $d, d' \in
{\rm cl}(\bigcup \mathcal{D})$ then
$C(d,g|f^{-1}(y))=C{(d',g|f^{-1}(y))}$ (note that $g({\rm cl}(
\bigcup \mathcal{D}))$ is a singleton). Hence $C(d,g|f^{-1}(y))
\subset B(L_1,1/2)$ for each $d \in {\rm cl}( \bigcup \mathcal{D})$,
and this implies cl$( \bigcup \mathcal{D}) \in \mathcal{D}$. Then
cl$(\bigcup \mathcal{D})$ is a maximal element of $\mathcal{D}$.
 If cl$(\bigcup \mathcal{D}) \neq L_1$, then by ($\sharp$) there exists
a proper subcontinuum $D' \subset L_1$ such that $D'$ contains
cl$(\bigcup \mathcal{D})$ as a proper subcontinuum of $D'$ and
 $C(d,g|f^{-1}(y)) \subset B(L_1,1/2)$
for each $d \in D'$. But this contradicts  the fact that cl$(
\bigcup \mathcal{D})$
 is a maximal element of
$\mathcal{D}$. So cl$(\bigcup \mathcal{D})=L_1$. But this is a
contradiction because diam$g(L_1)>0$ and $g({\rm cl}(\bigcup
\mathcal{D}))$ is a singleton. So there exists $D^* \in \mathcal{D}$
such that diam$g(D^*) > 0$. This completes the proof of claim.

\vspace{2mm}

Let $L_2=D^*$. Then  $L_2$ has the required property. By induction,
we can find a  decreasing sequence $\{L_k\}_{k=1}^{\infty}$ of
subcontinua of $L$ such that for any $k \in \mathbb{N}$ we have

\begin{enumerate}
\item[($*$)] diam$g(L_k)>0$;
\item[($**$)] $C(x,g|f^{-1}(y)) \subset B(L_k,1/(k+1))$ for each $x \in
L_{k+1}$.
\end{enumerate}

It is easy to see that $C(x,g|f^{-1}(y)) \subset L$ for each $x \in
\bigcap_{k=1}^{\infty}L_k$. This implies $g \in \mathcal{K}(H)$,
which completes the proof.
\end{proof}

Obviously, if $Y=\bigcup_{m=1}^{\infty}Y_m$,
$\mathcal{K}(Y)=\bigcap_{i,m=1}^{\infty}\mathcal{K}(Y_m)$.
Therefore, according to Proposition 2.1, it suffices to show that
$\mathcal{K}(m,n,H)$ is open and dense in $C(X,M)$ with respect to
the source limitation topology for $m,n\geq 1$ and any closed
$H\subset Y$ in the following cases: (i) $H$ is finite-dimensional
and $M$ a Krasinkiewicz $ANR$-space; (ii) $H$ is a $C$-space and $M$
a Krasinkiewicz space homeomorphic to a closed convex subset of a
Banach space.

In both of the above two cases we follow the scheme from the proof
of \cite[Theorem 1.1]{v}. In particular, we need the following lemma
established in \cite[Lemma 2.1]{v}.

\begin{lem}\cite{v} Every complete $ANR$-space $M'$ admits
a complete metric $\varrho$ generating its topology satisfying the
following condition: If $Z$ is a paracompact space, $A\subset Z$ a
closed set and $\varphi\colon Z\to M'$ a map, then for every
function $\alpha:Z\to(0,1]$ and every map $g\colon A\to M'$ with
$\varrho\big(g(z),\varphi(z)\big)<\alpha(z)/8$ for all $z\in A$,
there exists a map $\bar{g}\colon Z\to M'$ extending $g$ such that
$\varrho\big(\bar{g}(z),\varphi(z)\big)<\alpha(z)$ for all $z\in Z$.
\end{lem}

\subsection{Proof that $\mathcal{K}(m,n,H)$ is open in $C(X,M)$ for any
$m,n\geq1$ and any closed $H\subset Y$} In this subsection we prove
that all sets $\mathcal{K}(m,n,H)$ are open in $C(X,M)$, where
$(M,\varrho)$ is a complete metric (not necessarily an $ANR$ or a
Krasinkiewicz) space.

\begin{lem} Let $g\in
\mathcal{K}(m,n,y)$ for some $y\in Y$ and $m,n\geq 1$. Then there
exists a neighborhood $V_y$ of $y$ in $Y$ and $\delta_y>0$ such that
$y'\in V_y$ and $\varrho\big(g_1(x),g(x)\big)<\delta_y$ for all
$x\in f^{-1}(y')$ yields $g_1\in\mathcal{K}(m,n,y')$.
\end{lem}

\begin{proof}
Indeed, otherwise we can find a local base $\{V_k\}_{k \in
\mathbb{N}}$ of neighborhoods of $y$ in $Y$, points $y_k \in V_k$
and maps $g_k \in C(X,M)$ such that $\varrho(g_k(x), g(x)) < 1/k$
for all $x \in f^{-1}(y_k)$ but $g_k$ does not belong to
$K(m,n,y_k)$. Consequently, for every $k$ there exists a continuum
$F_k \subset f^{-1}(y_k)$ such that diam$g_k(F_k) \ge 1/n$ and
$C(x,g_k|f^{-1}(y_k))$ is not a subset of $B(F_k, 1/m)$ for any $x
\in F_k$. Then all $F_k$ are contained in the compact set
$P=f^{-1}(\{y_k\}_{k \in \mathbb{N}} \cup \{y\})$. We may assume
that $\{F_k\}_{k \in \mathbb{N}}$ converges to a continuum $F$. It
follows that $F \subset f^{-1}(y)$ and diam$g(F) \ge 1/n$. Since $g
\in K(m,n,y)$ there exists $t \in F$ such that $C(t,g|f^{-1}(y))
\subset B(F, 1/m)$. Since lim$F_k=F$, for each $k$ there exists $t_k
\in F_k$ with lim$t_k=t$. We may assume that $\{C(t_k,
g_k|f^{-1}(y_k))\}_{k \in \mathbb{N}}$ converges to a continuum $C$.
Note that $C \subset C(t, g|f^{-1}(y))$. Since
$C(t_k,g_k|f^{-1}(y_k))\backslash B(F_k,1/m)\neq\varnothing$, it is
easy to see that $C$ is not contained in $B(F,1/m)$. This is a
contradiction.
\end{proof}

Now, we are in a position to show that the sets $\mathcal{K}(m,n,H)$
are open in $C(X,M)$.

\begin{pro}\label{open} For
any closed $H\subset Y$ and any $m,n\geq 1$, the set
$\mathcal{K}(m,n,H)$ is open in $C(X,M)$ with respect to the source
limitation topology.
\end{pro}

\begin{proof} Let $g_0 \in K(m,n,H)$. Then, by Lemma 2.3,
for every $y \in H$ there exist a neighborhood $V_y$ and a positive
$\delta_y \le 1$ such that $g \in K(m,n,y')$ provided $g|f^{-1}(y')$
is $\delta_y$-closed to $g_0|f^{-1}(y')$. The family $\{V_y \cap H:
y \in H \}$ can be supposed to be locally finite in $H$. Consider
the set-valued lower semi-continuous map $\psi:H \to (0,1]$,
$\psi(y)= \bigcup \{ (0,\delta_z]: y \in V_z\}$. By \cite[Theorem
6.2, p.116]{rs}, $\psi$ admits a continuous selection $\beta: H \to
(0,1]$. Let $\overline{\beta} : Y \to (0,1]$ be a continuous
extension of $\beta$ and $\alpha = \overline{\beta} \circ f$. It
remains only to show that if $g \in C(X,M)$ with
$\varrho(g_0(x),g(x)) < \alpha(x)$ for all $x \in X$, then $g \in
K(m,n,y)$ for all $y\in H$. So, we take such a $g$ and fix $y \in
H$. Then there exists $z \in H$ with $y \in V_z$ and $\alpha(x) \le
\delta_z$ for all $x \in f^{-1}(y)$. Hence, $\varrho(g(x),g_0(x))<
\delta_z$ for each $x \in f^{-1}(y)$. According to the choice of
$V_z$, $g \in K(m,n,y)$. This completes the proof.
\end{proof}

\subsection{$\mathcal{K}(m,n,H)$ is dense in $C(X,M)$ for finite-dimensional $H$}
In this subsection we show that $\mathcal{K}(m,n,H)$ is dense in
$C(X,M)$ with respect to the source limitation topology provided
$H\subset Y$ is a closed finite-dimensional subset and $M$ a
Krasinkiewicz complete $ANR$-space.  We need to show that
$B_\varrho(g,\e)=\{g'\in C(X,M):\varrho(g,g')<\e\}$ meets
$\mathcal{K}(m,n,H)$ for every $g\in C(X,M)$ and every continuous
function $\e\colon X\to (0,1]$, where $\varrho$ is a complete metric
on $M$ satisfying the hypotheses of Lemma 2.2. To this end, fix
$g_0\in C(X,M)$ and $\e\in C(X,(0,1/64])$. Consider the set-valued
map $\Phi_{\e}\colon Y\to C(X,M)$,
$\Phi_\e(y)=\mathcal{K}(m,n,y)\cap B_\varrho(g_0,\e)$, where
$C(X,M)$ carries the compact open topology.

\begin{lem} Let $y_0\in Y$ and $\Phi_\e(y_0)$ contain a compact set $K$.
Then there exists a neighborhood $V(y_0)$ of $y_0$ such that
$K\subset\Phi_\e(y)$ for every $y\in V(y_0)$.
\end{lem}

\begin{proof}
Suppose there exists a sequence $\{y_j\}_{j\geq 1}$ converging to
$y_0$ in Y such that $K\backslash\Phi_\e(y_j)\neq\emptyset$. Let
$g_j\in K\backslash\Phi_\e(y_j)$, $j\geq 1$, and
$P=f^{-1}(\{y_0\}\cup\{y_j\}_{j\geq 1})$. The restriction map
$\pi_P\colon C(X,M)\to C(P,M)$ is continuous when both $C(X,M)$ and
$C(P,M)$ are equipped with the compact open topology. Moreover, the
compact open topology on $C(P,M)$ coincides with the uniform
convergence. Hence, there exists a subsequence $\{g_{j_k}\}$ of
$\{g_j\}$ such that $\pi_P(g_{j_k})$ converges to $\pi_P(g)$ in
$C(P,M)$ for some $g\in K$. Since $g\in\mathcal{K}(m,n,y_0)$, we can
apply Lemma 2.3 to find a neighborhood $V$ of $y_0$ in $Y$ and a
positive $\delta>0$ such that $y'\in V$ and
$\varrho\big(g(x),g'(x)\big)<\delta$ for all $x\in f^{-1}(y')$
implies $g'\in\mathcal{K}(m,n,y')$. Now, choose $j_k$ with
$y_{j_k}\in V$ and $\varrho\big(g(x),g_{j_k}(x)\big)<\delta$ for any
$x\in f^{-1}(y_{j_k})$. Then $g_{j_k}\in\mathcal{K}(m,n,y_{j_k})$.
So, $g_{j_k}\in\Phi_\e(y_{j_k})$ which contradicts the choice of the
functions $g_j$.
\end{proof}

\begin{lem} Every
$\Phi_\e(y)$ has the following property: If
$\hat{v}\colon\mathbb{S}^k\to\Phi_\e(y)$ is continuous, where $k\geq
0$ and $\mathbb{S}^k$ is the $k$-sphere, then $\hat{v}$ can be
extended to a continuous map
$\displaystyle\hat{u}\colon\mathbb{B}^{k+1}\to\Phi_{64\e}(y)$.
\end{lem}

\begin{proof} Let us mention
the following property of the function space $C(X,M)$ with the
compact open topology: For any metrizable space $Z$ a map
$\hat{w}\colon Z\to C(X,M)$ is continuous if and only if the map
$w\colon Z\times X\to M$, $w(z,x)=\hat{w}(z)(x)$, is continuous.
Hence, every map $\hat{v}\colon\mathbb{S}^k\to\Phi_\e(y)$ generates
a continuous map $v\colon\mathbb{S}^k\times X\to M$ defined by
$v(z,x)=\hat{v}(z)(x)$ such that
$\varrho\big(v(z,x),g_0(x)\big)<\e(x)$ for all
$(z,x)\in\mathbb{S}^k\times X$.



Let $\pi_y\colon C(X,M)\to C(f^{-1}(y),M)$ be the restriction map.
 It is easily seen that $\pi_y$ is continuous and
open when both $C(X,M)$ and $C(f^{-1}(y),M)$ are equipped with the
source limitation or the compact open topology. Since $f^{-1}(y)$ is
compact, the source limitation, the compact open and the uniform
convergence topologies on $C(f^{-1}(y),M)$ coincide. Therefore,
$\pi_y\big(\mathcal{K}(m,n,y)\big)$ is open in $C(f^{-1}(y),M)$ and
contains the compact set $\pi_y\big(\hat{v}(\mathbb{S}^k)\big)$.
Consequently, the distance (in the space $C(f^{-1}(y),M)$) between
$\pi_y\big(\hat{v}(\mathbb{S}^k)\big)$ and
$C(f^{-1}(y),M)\backslash\pi_y\big(\mathcal{K}(m,n,y)\big)$ is
positive. Denote this distance by $\delta_1$.

Obviously
$\delta_2=\inf\{\e(x)-\varrho\big(v(z,x),g_0(x)\big):(z,x)\in\mathbb{S}^k\times
f^{-1}(y)\}$ is positive. According to Lema 2.2, there exists a
continuous extension $v_1\colon\mathbb{B}^{k+1}\times f^{-1}(y)\to
M$ of the map $v|\big(\mathbb{S}^k\times f^{-1}(y)\big)$ with
$\varrho\big(v_1(z,x),g_0(x)\big)<8\e(x)$ for all
$(z,x)\in\mathbb{B}^{k+1}\times f^{-1}(y)$. Let
$\delta_3=\inf\{8\e(x)-\varrho\big(v_1(z,x),g_0(x)\big):(z,x)\in\mathbb{B}^{k+1}\times
f^{-1}(y)\}$. Since $M$ is a Krasinkiewicz space, there exists a
Krasinkiewicz map $v_2\colon\mathbb{B}^{k+1}\times f^{-1}(y)\to M$
such that $\varrho\big(v_2(z,x),v_1(z,x)\big)<\delta/8$ for all
$(z,x)\in\mathbb{B}^{k+1}\times f^{-1}(y)$, where
$\delta=\min\{\delta_1,\delta_2,\delta_3\}$. Therefore, we have a
map $\hat{v}_2\colon\mathbb{B}^{k+1}\to C(f^{-1}(y),M)$. The choice
of $\delta_3$ implies
$$\varrho\big(v_2(z,x),g_0(x)\big)<8\e(x)\leqno{(1)}$$
for all $(z,x)\in\mathbb{B}^{k+1}\times f^{-1}(y)$. Moreover, $v_2$
being a Krasinkiewicz map yields that all maps $\hat{v}_2(z)\colon
f^{-1}(y)\to M$, $z\in\mathbb{B}^{k+1}$, are also Krasinkiewicz. On
the other hand, by Lemma 2.2 and (1), every $\hat{v}_2(z)$ can be
extended to a map from $X$ into $M$. Therefore,
$$\hat{v}_2\big(\mathbb{B}^{k+1}\big)\subset\pi_y\big(\mathcal{K}(m,n,y)\big)
\leqno{(2)}.$$

Representing the ball $\mathbb{B}^{k+1}$ as a cone with a base
$\mathbb{S}^k$ and a vertex $z_0$, we can consider $v_2$ as a
homotopy from $\mathbb{S}^k\times f^{-1}(y)\times [0,1]$ into $M$
between the maps $v_2|\big(\mathbb{S}^k\times
f^{-1}(y)\times\{0\}\big)$ and $v_2|\big(\{z_0\}\times
f^{-1}(y)\big)$. Observe also that
$\varrho\big(v_2(z,x,0),v(z,x)\big)<\delta/8$ for any
$(z,x)\in\mathbb{S}^k\times f^{-1}(y)$. Hence, the map
$\varphi\colon\mathbb{S}^k\times f^{-1}(y)\times\{0,1\}\to M$,
$$\varphi(z,x,t)=
\begin{cases}
v(z,x)&\mbox{if $t=0$};\\
v_2(z,x,0)&\mbox{if $t=1$}.\\
\end{cases},$$
is $(\delta/8)$-close to $v$. Consequently, by Lemma 2.2, $\varphi$
admits a continuous extension $v_3\colon \mathbb{S}^k\times
f^{-1}(y)\times [0,1]\to M$ such that
$\varrho\big(v_3(z,x,t),v(z,x)\big)<\delta$ for every
$(z,x,t)\in\mathbb{S}^k\times f^{-1}(y)\times [0,1]$. Since
$\delta<\min\{\delta_1,\delta_2\}$, for any
$(z,x,t)\in\mathbb{S}^k\times f^{-1}(y)\times [0,1]$ we have
$$\varrho\big(v_3(z,x,t),v(z,x)\big)<\delta_1,
\leqno{(3)}$$ and
$$\varrho\big(v_3(z,x,t),g_0(x)\big)<\e(x).
\leqno{(4)}$$

Therefore, $v_3$ is a homotopy connecting the maps $v$ and
$v_2|\big(\mathbb{S}^k\times f^{-1}(y)\times\{0\}\big)$, while $v_2$
is a homotopy connecting the maps $v_2|\big(\mathbb{S}^k\times
f^{-1}(y)\times\{0\}\big)$ and $v_2|\big(\{z_0\}\times
f^{-1}(y)\big)$. Combining these two homotopies, we obtain a map
$u_1\colon\mathbb{S}^k\times f^{-1}(y)\times [0,1]\to M$ such that
$u_1(z,x,0)=v(z,x)$, $u_1(z,x,1)=v_2(z_0,x)$ and
$\varrho\big(u_1(z,x,t),g_0(x)\big)<8\e(x)$ for all
$(z,x,t)\in\mathbb{S}^k\times f^{-1}(y)\times [0,1]$. Obviously,
$u_1$ can also be considered as a map from $\mathbb{B}^{k+1}\times
f^{-1}(y)$ into $M$ such that $u_1|\big(\mathbb{S}^{k}\times
f^{-1}(y)\big)=v$ and $\varrho\big(u_1(z,x),g_0(x)\big)<8\e(x)$,
$(z,x)\in\mathbb{B}^{k+1}\times f^{-1}(y)$. Now consider the map
$u_2\colon\big(\mathbb{B}^{k+1}\times
f^{-1}(y)\big)\cup\big(\mathbb{S}^{k}\times X\big)\to M$ with
$u_2|\big(\big(\mathbb{B}^{k+1}\times f^{-1}(y)\big)=u_1$ and
$u_2|\big(\big(\mathbb{S}^{k}\times X\big)=v$. Finally, using Lemma
2.2, we extend $u_2$ to a map $u\colon\mathbb{B}^{k+1}\times X\to M$
such that
$$\varrho\big(u(z,x),g_0(x)\big)<64\e(x)\leqno{(5)}$$
for any $(z,x)\in\mathbb{B}^{k+1}\times X$. Then
$\hat{u}\colon\mathbb{B}^{k+1}\to C(X,M)$ extends the map $\hat{v}$.
Moreover, (2), (3) and the choice of $\delta_1$ implies that
$\hat{u}\big(\mathbb{B}^{k+1}\big)\subset\mathcal{K}(m,n,y)$. On the
other hand, (5) yields $\hat{u}\big(\mathbb{B}^{k+1}\big)\subset
B_\varrho(g_0,64\e)$. Hence,
$\hat{u}\big(\mathbb{B}^{k+1}\big)\subset\Phi_{64\e}(y)$.
\end{proof}

Next proposition completes the proof of Theorem 1.1 in the case $Y$
is strongly countable-dimensional.

\begin{pro} Let $H\subset Y$ be a closed finite-dimensional set.
Then $\mathcal{K}(m,n,H)$, $m,n\geq 1$, are dense sets in $C(X,M)$
with respect to the source limitation topology.
\end{pro}

\begin{proof} Let
$\dim H\leq k$. Define the set-valued maps $\Phi_j\colon H\to
C(X,M)$, $j=0,..,k$,
$\displaystyle\Phi_j(y)=\Phi_{\e/8^{2(k-j)+1}}(y)$. Obviously,
$\Phi_0(y)\subset\Phi_1(y)\subset...\subset\Phi_k(y)=\Phi_{\e/8}(y)$.
According to Lemma 2.6, every map from $\mathbb{S}^k$ into
$\Phi_j(y)$ can be extended to a map from $\mathbb{B}^{k+1}$ into
$\Phi_{j+1}(y)$, where $j=0,1,..,k-1$ and $y\in H$. Moreover, by
Lemma 2.5, any $\Phi_j(y)$ has the following property: if
$K\subset\Phi_j(y)$ is compact, then there exists a neighborhood
$V_y$ of $y$ in $Y$ such that $K\subset\Phi_j(z)$ for all $z\in
V_y\cap H$. So, we may apply \cite[Theorem 3.1]{gu} to find a
continuous selection $\theta\colon H\to C(X,M)$ of $\Phi_k$. Hence,
$\theta(y)\in\Phi_{\e/8}(y)$ for all $y\in H$. Now, consider the map
$g\colon f^{-1}(H)\to M$, $g(x)=\theta(f(x))(x)$. Using that
$C(X,M)$ carries the compact open topology, one can show that $g$ is
continuous. Moreover, $\varrho\big(g(x),g_0(x)\big)<\e(x)/8$ for all
$x\in f^{-1}(H)$. Then, by Lemma 2.2, $g$ can be extended to a
continuous map $\bar{g}\colon X\to M$ with
$\varrho\big(\bar{g}(x),g_0(x)\big)<\e(x)$, $x\in X$. It follows
from the definition of $g$ that $g|f^{-1}(y)=\theta(y)|f^{-1}(y)$
for every $y\in H$. Since $\theta(y)\in\mathcal{K}(m,n,y)$ for all
$y\in H$,  $\bar{g}\in\mathcal{K}(m,n,H)$. Hence,
$B_\varrho(g_0,\e)\cap\mathcal{K}(m,n,H)\neq\emptyset$.
\end{proof}

\subsection{$\mathcal{K}(m,n,H)$ is dense in $C(X,M)$ for $H$ being a $C$-space}
We now turn to the proof of Theorem 1.1 in the case $Y$ is a
$C$-space and $M$ a Krasinkiewicz space homeomorphic to a closed
convex subset $M'$ of a given Banach space $E$. Suppose $M=M'$ and
let $\varrho$ be the metric on $M$ inherited from the norm of $E$
and $\Psi_\e\colon Y\to C(X,M)$ be the set-valued map
$\Psi_\e(y)=\overline{B}_\varrho(g_0,\e)\cap\mathcal{K}(m,n,y)$,
where $C(X,M)$ is equipped again with the compact open topology and

$$\overline{B}_\varrho(g_0,\e)=\{g\in
C(X,M):\varrho\big(g_0(x),g(x)\big)\leq\e(x)\hbox{~}\mbox{for
all}\hbox{~}x\in X\}.$$

\begin{lem}
$\Psi_\e$ has the following property: Every map
$\hat{v}\colon\mathbb{S}^k\to\Psi_\e(y)$, $n\geq 0$, can be extended
to a map $\displaystyle\hat{u}\colon\mathbb{B}^{k+1}\to\Psi_\e(y)$.
\end{lem}

\begin{proof} All function
spaces in this proof are equipped with the compact open topology.
Let $\pi_y\colon C(X,M)\to C(f^{-1}(y),M)$ be the restriction map
and
$P(y)=\overline{B}_\varrho(g_0,\e,y)\backslash\pi_y\big(\mathcal{K}(m,n,y)\big)$,
where $\overline{B}_\varrho(g_0,\e,y)$ is the set
$$\{g\in
C(f^{-1}(y),M):\varrho\big(g_0(x),g(x)\big)\leq\e(x)\hbox{~}\mbox{for
all}\hbox{~}x\in f^{-1}(y)\}.$$  Since
$\pi_y\big(\mathcal{K}(m,n,y)$ is open in $C(f^{-1}(y),M)$,
$P(y)\subset\overline{B}_\varrho(g_0,\e,y)$ is closed.

We are going to show that $P(y)$ is a $Z$-set in
$\overline{B}_\varrho(g_0,\e,y)$, i.e., every map $\hat{w}\colon
K\to\overline{B}_\varrho(g_0,\e,y)$,  where $K$ is compact, can be
approximated by a map $\hat{w}_1\colon
K\to\overline{B}_\varrho(g_0,\e,y)\backslash
P(y)=\overline{B}_\varrho(g_0,\e,y)\cap\pi_y\big(\mathcal{K}(m,n,y)\big)$.
To this end, fix $\delta>0$ and let $w\colon K\times f^{-1}(y)\to M$
be the map generated by $\hat{w}$. So,
$\varrho\big(w(z,x),g_0(x)\big)\leq\e(x)$ for all $(z,x)\in K\times
f^{-1}(y)$. Since $f^{-1}(y)$ is compact, there exists $\lambda\in
(0,1)$ such that $\lambda\max\{\e(x):x\in f^{-1}(y)\}<\delta/2$.
Define the map $w_1\colon K\times f^{-1}(y)\to M$ by
$w_1(z,x)=(1-\lambda)w(z,x)+\lambda g_0(x)$. Then, for all $(z,x)\in
K\times f^{-1}(y)$ we have
$$\varrho\big(w_1(z,x),w(z,x)\big)\leq\lambda\e(x)<\delta/2$$
and
$$\varrho\big(w_1(z,x),g_0(x)\big)\leq(1-\lambda)\e(x)<\e(x).$$
Since $M$ is a Krasinkiewicz space, there exists a Krasinkiewicz map
$w_2\colon K\times f^{-1}(y)\to M$ which is $\delta_1$-close to
$w_1$, where $\delta_1=\min\{\lambda\e(x):x\in f^{-1}(y)\}$. Hence,
for every $(z,x)\in K\times f^{-1}(y)$ we have
$$\varrho\big(w_2(z,x),g_0(x)\big)\leq\e(x)\hbox{~}\mbox{and}\hbox{~}
\varrho\big(w_2(z,x),w(z,x)\big)<\delta.$$ The last two inequalities
imply that the map $\hat{w}_2\colon K\to C(f^{-1}(y),M)$ is
$\delta$-close to $\hat{w}$ and
$\hat{w}_2(K)\subset\overline{B}_\varrho(g_0,\e,y)$. Moreover, every
$\hat{w}_2(z)$, $z\in K$, being a map from $f^{-1}(y)$ into $M$, can
be extended to a map from $X$ to $M$ because $M$ is a closed convex
subset of $E$. Since $w_2$ is a Krasinkiewicz map, so are the maps
$\hat{w}_2(z)$, $z\in K$. Hence,
$\hat{w}_2(K)\subset\pi_y\big(\mathcal{K}(m,n,y)\big)$. So, $P(y)$
is a $Z$-set in $\overline{B}_\varrho(g_0,\e,y)$.

Let us complete the proof of the lemma. For every map
$\hat{v}\colon\mathbb{S}^k\to\Psi_\e(y)$ the composition
$\pi_y\circ\hat{v}$ is a map from $\mathbb{S}^k$ into
$\overline{B}_\varrho(g_0,\e,y)\cap\pi_y\big(\mathcal{K}(m,n,y)\big)$.
Since $P(y)$ is a $Z$-set in the convex set
$\overline{B}_\varrho(g_0,\e,y)$, by \cite[Proposition 6.3]{vu:98},
there exists a map
$\hat{v}_1\colon\mathbb{B}^{k+1}\to\overline{B}_\varrho(g_0,\e,y)\cap\pi_y\big(\mathcal{K}(m,n,y)\big)$
extending $\pi_y\circ\hat{v}$. Consider the map $v_2\colon A\to M$,
where $A=\big(\mathbb{B}^{k+1}\times
f^{-1}(y)\big)\cup\big(\mathbb{S}^k\times X\big)$, defined by
$v_2|\big(\mathbb{B}^{k+1}\times f^{-1}(y)\big)=v_1$ and
$v_2|\big(\mathbb{S}^k\times X\big)=v$. Next, take a selection
$u\colon\mathbb{B}^{k+1}\times X\to M$ for the set-valued map
$\phi\colon\mathbb{B}^{k+1}\times X\to M$, $\phi(z,x)=v_2(z,x)$ if
$(z,x)\in A$ and $\phi(z,x)=$cl$\big(B_\varrho(g_0(x),\e(x))\big)$
if $(z,x)\not\in A$. Such $u$ exists by Michael's \cite{mi}
convex-valued selection theorem. Obviously $u$ extends $v_2$ and
$\varrho\big(u(z,x),g_0(x)\big)\leq\e(x)$ for every
$(z,x)\in\mathbb{B}^{k+1}\times X$.  Finally, observe that $\hat{u}$
is the required extension of $\hat{v}$.
\end{proof}

We can finish the proof of Theorem 1.1.

\begin{pro}
Suppose $H\subset Y$ is a closed $C$-space and $M$ a closed convex
subset of a Banach space $E$. Then the sets $\mathcal{K}(m,n,H)$,
$m,n\geq 1$, are dense in $C(X,M)$ with respect to the source
limitation topology.
\end{pro}

\begin{proof} Consider the set-valued map $\Psi_\e\colon H\to C(X,M)$. It
follows from the proof of Lemma 2.5 that if $K\subset\Psi_\e(y_0)$
for some compactum $K$ and $y_0\in H$, then $y_0$ admits a
neighborhood $V\subset H$ with $K\subset\Psi_\e(y)$ for all $y\in
V$. Moreover, according to Lemma 2.8, every image $\Psi_\e(y)$ is
aspherical, i.e., any map from $\mathbb{S}^k$ into $\Psi_\e(y)$,
$k\geq 0$, can be extended to a map from $\mathbb{B}^{k+1}$ to
$\Psi_\e(y)$. Then, by the Uspenskij selection theorem \cite[Theorem
1.3]{vu:98}, $\Psi_\e$ admits a continuous selection $\theta\colon
H\to C(X,M)$. Repeating the arguments from the proof of Proposition
2.7, we obtain a map $g\colon f^{-1}(H)\to M$ such that
$\varrho\big(g(x),g_0(x)\big)\leq\e(x)$ for every $x\in f^{-1}(H)$
and $g|f^{-1}(y)=\theta(y)|f^{-1}(y)$, $y\in H$. Applying once more
the Michael \cite{mi} convex-valued selection theorem  for the
set-valued map $\vartheta\colon X\to M$, $\vartheta(x)=g(x)$ if
$x\in f^{-1}(H)$ and
$\vartheta(x)=\overline{B}_\varrho\big(g_0(x),\e(x)\big)$ if
$x\not\in f^{-1}(H)$, we obtain a selection $\bar{g}$ for
$\vartheta$. Obviously, $\bar{g}$ extends $g$ and
$\bar{g}\in\overline{B}_\varrho(g_0,\e)$. Since
$\theta(y)\in\mathcal{K}(m,n,y)$ for all $y\in H$, we have
$\bar{g}\in\overline{B}_\varrho(g_0,\e)\cap\mathcal{K}(m,n,H)$.
Hence, $\mathcal{K}(m,n,H)$ is dense in $C(X,M)$.
\end{proof}

\section{Some properties of Krasinkiewicz spaces}

In this section we investigate the class of Krasinkiewicz spaces
and, on that base, provide more spaces from this class. Let us start
with the following proposition whose proof is straightforward.

\begin{pro} For every  space $M$ we have:
\begin{enumerate}
\item If $M$ is a Krasinkiewicz space, then so is any open subset of $M$;
\item If every compact set in $M$ is contained in a Krasinkiewicz subset of $M$,
then $M$ is also a Krasinkiewicz space.
\end{enumerate}
\end{pro}

\begin{cor}
Every polyhedron is a Krasinkiewicz space.
\end{cor}

\begin{proof}
Apply Proposition 3.1(2) and the fact that each compact polyhedron
is a Krasinkiewicz space \cite{m}.
\end{proof}

Next proposition is an analogue of \cite[Theorem 4.2]{st}.

\begin{pro} Suppose $M$ is completely metrizable and for every
$\varepsilon>0$ there exist a Krasinkiewicz space $Z_\varepsilon$
and maps $r\colon M\to Z_\varepsilon$ and $\phi\colon
Z_\varepsilon\to M$ such that $\phi$ is light and $\phi\circ r$ is
$\varepsilon$-close to the identity on $M$. Then $M$ is a
Krasinkiewicz space.
\end{pro}

\begin{proof} Let $g\in C(X,M)$ and $\varepsilon>0$, where $X$ is
compact. Then there exists a Krasinkiewicz space $Z_{\varepsilon/2}$
and two maps $r\colon M\to Z_{\varepsilon/2}$, $\phi\colon
Z_{\varepsilon}\to M$ such that $\phi$ is light and $\phi\circ r$ is
$\varepsilon/2$-close to the identity on $M$. Take $\delta>0$ and a
neighborhood $U$ of $r(g(X))$ in $Z_{\varepsilon/2}$ such that
$dist(\phi(z_1),\phi(z_2))<\varepsilon/2$ provided $z_1, z_2\in U$
and $dist(z_1,z_2)<\delta$. Next, choose a Krasinkiewicz map
$h\colon X\to Z_{\varepsilon/2}$ which is $\delta$-close to $r\circ
g$ and $h(X)\subset U$. Finally, $g'=\phi\circ h$ is
$\varepsilon$-close to $g$ and, since $\phi$ is light, $g'$ is a
Krasinkiewicz map (see \cite[Proposition 3.1]{m}).
\end{proof}

Proposition 3.3 is of special interest when all $Z_\varepsilon$ are
subsets of $M$ and the maps $r$ are retractions (in such a case we
say that $M$ {\em admits small retractions to Krasinkiewicz
spaces}). Since every compact Menger manifold (a manifold modeled on
the Menger cube $\mu^n$ for some $n\geq 1$), as well as every
$1$-dimensional Peano continuum, admits small retractions to compact
polyhedra, it was observed in \cite[Theorem 3.2-3.3]{m} that any
such a space is Krasinkiewicz. Moreover, every N\"obeling manifold
also admits small retractions to polyhedra, see \cite{ckt}. So, by
Proposition 3.3, we have:
\begin{cor}
Each of the following are Krasinkiewicz spaces: $1$-dimensional
Peano continua, Menger manifolds and N\"obeling manifolds.
\end{cor}

\begin{pro}
A product of finitely many Krasinkiewicz spaces is a Krasinkiewicz
space.
\end{pro}

\begin{proof}
We need to prove the proposition for a product of two Krasinkie-wicz
spaces $M_1$ and $M_2$. In this case, the proof is reduced to show
that if $X$ is a metric compactum and $g_i\colon X\to M_i$, $i=1,2$,
are Krasinkiewicz maps, then the product map $g=g_1\triangle
g_2\colon X\to M_1\times M_2$ is also a Krasinkiewicz map. And that
easily follows.
\end{proof}

Some more examples of Krasinkiewicz spaces are provided by next
theorem.

\begin{thm}\label{thm:local}
A complete $ANR$-space $M$ is a Krasinkiewicz space if and only if
it has an open cover of Krasinkiewicz spaces.
\end{thm}

Proof. It suffices to show that $M$ is Krasinkiewicz if each $y \in
M$ has a neighborhood $U_y$ in $M$ which is a Krasinkiewicz space.
We fix a compactum $X$ and choose $\varepsilon_y >0$, $y\in M$, with
$B(y,3 \varepsilon_y) \subset U_y$. Let $H_y$ be the set of all maps
$g:X \to M$ satisfying next condition:

\begin{enumerate}
\item[$(a)$] If  $L\subset X$ is a subcontinuum such that diam $g(L)>0$ and
$g(L)\subset$ cl$(B(y, \varepsilon_y))$, then there exists $x \in L$
with $C(x,g) \subset L$.
\end{enumerate}

Now, for each $m,n \in \mathbb{N}$ consider the set
$H_{m,n,y}\subset C(X,M)$ of all maps $g$ such that:

\begin{enumerate}
\item[$(b)$] If $L\subset X$ is continuum with diam $g(L) \ge 1/n$ and
$g(L) \subset$ cl$(B(y, \varepsilon_y))$, then $C(x,g) \subset
B(L,1/m)$ for some $x \in L$.
\end{enumerate}

\textit{ Claim $1$. $H_y= \bigcap_{m,n \in \mathbb{N}} H_{m,n,y}$}.

The proof of this claim is similar to the proof of Proposition 2.1,
so it is omitted.

\medskip
\textit{Claim $2$. Every $H_{m,n,y}$ is open in $C(X,M)$}.

Let $f\in {\rm cl} (C(X,M) \setminus H_{m,n,y})$. Then there exists
a sequence of maps $\{f_i\}_{i=1}^{\infty} \subset C(X,M) \setminus
H_{m,n,y}$ with lim$f_i =f$. For each $i=1,2,\ldots,$ there exists a
subcontinuum $L_i \subset X$ such that diam$f_i(L_i) \ge 1/n$,
$f_i(L_i) \subset {\rm cl}(B(y, \varepsilon_y))$ and $C(x,f_i)$ is
not contained in $B(L_i,1/m)$ for each $x \in L_i$. We may assume
that $L_i$  converges to  a subcontinuum $L \subset X$. It is easy
to see that diam$f(L) \ge 1/n$ and $f(L) \subset {\rm cl}(B(y,
\varepsilon_y))$. Let $x \in L$ be arbitrary. Then $x$ is the limit
of a sequence $\{x_i\}_{i=1}^{\infty} \subset X$ such that $x_i \in
L_i$ for each $i =1,2, \ldots$. We may assume that $C(x_i,f_i)$
converges to a subcontinuum $C \subset X$. Since each $C(x_i, f_i)$
is not contained in $B(L_i,1/m)$, $C$ is not contained in
$B(L,1/m)$. Moreover, $x \in C \subset C(x,f)$. So, $f \in C(X,M)
\setminus H_{m,n,y}$. This completes the proof of Claim 2.

\medskip
\textit{Claim $3$. Every $H_y$ is dense in $C(X,M)$}.

Let $f \in C(X,M)$ and $\varepsilon>0$ with $\varepsilon <
\varepsilon_y$. Since $M$ is an $ANR$, there is a $\delta>0$ such
that each map $g:A \to M$, where $A\subset X$ is closed, has a
continuous extension $\hat{g}:X \to M$ which is $\varepsilon$-close
to $f$ provided $g$ is $\delta$-close to $f|A$. Since $U_y$ is a
Krasinkiewicz space and $f^{-1}({\rm cl}(B(y,2 \varepsilon_y))$ is
compact, there exists a Krasinkiewicz map $k:f^{-1}({\rm cl}(B(y,2
\varepsilon_y)) \to U_y$ such that $k$ is $\delta$-close to
$f|f^{-1}({\rm cl}(B(y,2 \varepsilon_y)))$. Then there exists a
continuous extension $\hat{k}:X \to Y$ of $k$ such that $\hat{k}$ is
$\varepsilon$-close to $f$. We are going to show that $\hat{k}\in
H_y$. Indeed, let $L$ be a subcontinuum of $X$ such that
diam$\hat{k}(L)
> 0$ and $\hat{k}(L) \subset {\rm cl}(B(y, \varepsilon_y))$. Then $L
\subset f^{-1}({\rm cl}(B(y,2 \varepsilon_y)))$. Since
$k:f^{-1}({\rm cl}(B(y,2 \varepsilon_y)) \to U_y$ is a Krasinkiewicz
map, there exists $x \in L$ such that $C(x,k) \subset L$. Note that
$C(x,k)=C(x,\hat{k})$ because $\hat{k}^{-1}(z)=k^{-1}(z)$ for each
$z \in {\rm cl}(B(y, \varepsilon_y))$. This completes the proof of
Claim 3.

Now, we can complete the proof of Theorem \ref{thm:local}. Let $f
\in C(X,M)$ and $\varepsilon>0$. Since $f(X)$ is compact, there
exist finitely many points $y_1,y_2,...,y_N \in f(X)$ such that
$f(X) \subset \bigcup_{i=1}^{N}B(y_i, 2^{-1}\varepsilon_{y_i})$. Let
$\delta_0={\rm min}\{\varepsilon, 2^{-1}\varepsilon_{y_1},
2^{-1}\varepsilon_{y_2},...,2^{-1}\varepsilon_{y_N} \}$. By previous
claims, $\bigcap_{i=1}^{N}H_{y_i}$ is a dense $G_\delta$-subset of
$C(X,M)$. So, we can find a map $g_0 \in \bigcap_{i=1}^{N}H_{y_i}$
$\delta_0$-close to $f$. It suffices to show that $g_0$ is a
Krasinkiewicz map. To this end, let $T$ be a subcontinuum of $X$
with diam$g_0(T)>0$. Note that $g_0(T) \subset
\bigcup_{i=1}^{N}B(y_i, \varepsilon_{y_i})$. Hence, there exists a
subcontinuum $T'\subset T$ and $j\in \{1,2,...,N\}$ such that
diam$g_0(T')>0$ and $g_0(T') \subset {\rm cl}(B(y_{j},
\varepsilon_{y_{j}}))$. Since $g_0 \in H_{y_{j}}$, there exists a
point $x_0 \in T'$ such that $C(x_0,g_0) \subset T' \subset T$. This
completes the proof.

Our final proposition provides spaces which are not Krasinkiewicz.
It implies, for example, that hereditarily indecomposable continua
can not be Krasinkiewicz spaces.

\begin{pro}
Let $Y$ be a non-degenerate continuum such that some open subset of
$Y$ contains no arc. Then the projection $p:Y \times\I \to Y$ can
not be approximated by Krasinkiewicz maps.
\end{pro}

\begin{proof} Let $U$ be an open subset of $Y$ such that $U$ contains
no arc. Choose a non-degenerate continuum $L \subset U$ and let
$\delta=$ diam$L$ and $\varepsilon=\min\{\delta/2, dist(L,X\setminus
U)\}$. We claim that every map $q: Y \times\I \to Y$ which is
$\varepsilon$-close to $p$ can not be Krasinkiewicz. Indeed, suppose
there exists such a Krasinkiewicz map $q_0$ and let $t \in\I$. Then
$q_0(L \times \{t\})$ is not a singleton, so there exists $y \in
q_0(L \times \{t\})$ and a component $C$ of $q_0^{-1}(y)$ such that
$C \subset L \times \{t\}$. Take any point $z \in p(C)$. Then
$q_0(\{z\} \times\I)$ is not a singleton. So $q_0(\{z\} \times\I)$
contains an arc. On the other hand, $q_0(\{z\} \times\I) \subset U$.
This is a contradiction.
\end{proof}


\end{document}